\documentclass[11pt,reqno]{amsart}

\usepackage[T1]{fontenc}     
\usepackage{lmodern}         
\usepackage{amsmath, amsthm, amssymb,amsaddr}
\usepackage{dcolumn}
\usepackage{bm}
\usepackage{amsmath} 
\usepackage{mathrsfs}
\usepackage{mathtools}
\usepackage{algpseudocode}
\usepackage{algorithm}
\usepackage[pdftex]{graphicx}
\usepackage{subfigure}
\usepackage{epstopdf}
\usepackage{tikz}
\usepackage{xcolor}
\usepackage[export]{adjustbox}
\usepackage{caption}
\usepackage{hyperref}
\usepackage[pagewise]{lineno}
\usepackage{multirow}
\usepackage{caption}
\usepackage{tikz}
\usepackage{subcaption}
\usepackage[export]{adjustbox} 

\theoremstyle{remark}

\newtheorem{lemma}{Lemma}

\DeclareMathOperator*{\argmin}{arg\,min}

\newif\ifblacktext
\blacktextfalse 

\ifblacktext
\fi

\newcommand{\be}[1]{\begin{equation}\label{#1}}
\newcommand{\ee}{\end{equation}}

\newcommand{\no}[1]{#1}

\newcommand{\bx}{\boldsymbol{x}}

\newcommand{\bn}{\boldsymbol{n}}

\newcommand{\bX}{\boldsymbol{X}}
\newcommand{\bt}{\mathfrak{t}}
\newcommand{\bu}{\boldsymbol{u}}
\newcommand{\bro}{\boldsymbol{\rho}}

\renewcommand{\no}[1]{} 

\title[A Reciprocity-Law-Compliant Photoacoustic Forward-Adjoint Operator]{A Reciprocity-Law-Compliant Photoacoustic Forward-Adjoint Operator}

\author{Ashkan Javaherian}
\address{Department of Bio-Electric, School of Electrical and Computer Engineering, University College of Engineering, University of Tehran, Tehran, Iran.}
\email{ajavherian62@gmail.com; ashkan.javaherian@ut.ac.ir}
\date{May 2026}

\begin{document} 
\maketitle

\begin{abstract}
We extend the forward--adjoint operator framework derived in our previous study to photoacoustic tomography (PAT). In that earlier work, the acoustic forward operator included a reception operator that maps, at each time step, the pressure wavefield in free space onto the boundary (receiver surface). It was shown that this reception operator serves as a left-inverse of an emission operator that maps the pressure restricted to the boundary (emitter surface) onto free space, perfectly complying with the reciprocity law of physics. In this study, we define the full PAT forward operator as a composite mapping composed of an acoustic forward operator equipped with a scaled variant of the previously proposed reception operator, and an operator describing the photoacoustic source. Singularities arising both in the reception step (due to the boundary restriction) and in the photoacoustic source (due to its instantaneous nature) are regularized using regularized Dirac delta distributions. The resulting PAT forward--adjoint operator pair satisfies an inner-product relation, which we verify through numerical experiments on a discretized domain. The effectiveness of the proposed operator pair is further demonstrated using an iterative minimization framework that yields both qualitatively and quantitatively accurate reconstructions of an initial pressure distribution from the corresponding Dirichlet-type boundary data.
\end{abstract}


\section{Introduction}

Photoacoustic tomography (PAT) is a hybrid imaging modality that combines the complementary advantages of optical and acoustic techniques. Optical illumination provides strong contrast due to the significant variability in optical absorption among biological tissues, while acoustic detection enables high spatial resolution because ultrasound waves experience relatively low scattering in tissue \cite{Xu2006}. 

In PAT, nanosecond electromagnetic pulses, typically in the visible or near-infrared spectral ranges, are used to irradiate the tissue. A portion of the delivered optical energy is absorbed by light-absorbing structures (chromophores) and partially converted into heat \cite{Rosencwaig1976,Li2009}. The resulting temperature rise at each location, which is proportional to the local optical absorption coefficient and photon density, induces thermoelastic expansion and leads to a rapid increase in pressure \cite{Li2009}. 

These locally generated pressure rises propagate outward as acoustic waves and are subsequently detected over time by ultrasound transducers placed on a detection surface surrounding the sample \cite{Xu2006,Li2009}.

The physics of photoacoustic tomography can be naturally separated into two stages. The first stage describes the generation of heat due to the absorption of optical energy, which gives rise to an initial pressure distribution \cite{Li2009,Bal2010}. The second stage describes the propagation of this initial pressure as acoustic waves \cite{Xu2006,Kruger1995}.

Accordingly, the inverse problem is typically decomposed into two parts. The acoustic inversion problem consists of reconstructing the initial pressure distribution from pressure measurements recorded in time on a detection surface \cite{Xu2006,Kruger1995}. The optical inversion problem consists of recovering the optical absorption coefficient from the reconstructed initial pressure distribution \cite{Li2009,Bal2010,Javaherian2019}. Our study, like the majority of work in the field, focuses on the acoustic inverse problem.

In Section~2, we review existing approaches for solving the acoustic step of the photoacoustic inverse problem and subsequently outline the contributions of our study to this active research area. Section~3 outlines the main formulas and algorithms to implement the forward and adjoint operators for the particular application of photoacoustic tomography. Section~4 presents numerical results for two-dimensional scenarios. This section begins with two numerical experiments designed to measure the accuracy of the derived adjoint operator under two configurations: (1) a simplified case in which receiver surfaces are assumed to be aligned with the sampled points on the computational grid, and (2) a more realistic configuration in which receivers are positioned on a circular boundary and therefore do not necessarily coincide with grid points. The accuracy of the derived adjoint operator is evaluated using an inner-product consistency test involving the forward operator and its adjoint. For each configuration, the test employs a randomly generated initial pressure source located within the receiver array, along with randomized boundary data. For the second receiver configuration, a full photoacoustic imaging scenario is set up and solved iteratively using our derived forward–adjoint operator pair within an iterative minimization framework. A brief discussion and concluding remarks are then provided in Section~5.

\section{Existing Approaches and Our Contributions}

This section first describes the existing approaches for solving the acoustic inverse problem in photoacoustic tomography (PAT), and then outlines our contributions to this area.

\subsection{Existing Approaches}

In this section, we provide a thorough review of analytical and numerical approaches for solving the acoustic inverse problem in PAT.

In \cite{Kruger1995}, the pressure time series recorded on the detection surface is described using an operator resembling a three-dimensional Radon transform of the heat source, where the integration is performed over spherical surfaces rather than planes. In \cite{Kostli2001}, an exact backprojection formula in the three-dimensional spatial frequency domain was proposed to reconstruct the initial source from the propagated pressure time series recorded on a planar surface. In \cite{Finch2004}, the heat source is reconstructed from its mean values over a family of spheres centered on the boundary, leading to a backprojection formula expressed in terms of a single-layer potential integral evaluated over the boundary. In \cite{Kunyanski2012}, fast one-step backprojection algorithms were proposed to reconstruct the initial source from its spherical mean values over spheres centered on the boundary. These algorithms were designed to reduce the computational cost associated with inversion formulas based on spherical means \cite{Finch2004}.

In \cite{Xu2005}, an exact reconstruction formula was proposed to recover the heat source from pressure time series generated by the source and sampled over time on planar, spherical, or cylindrical detection surfaces after propagation through the medium. This formula, known as the \textit{Universal Backprojection} (UBP) formula, can be expressed analytically as a time-reversed double-layer potential integral evaluated over the boundary \cite{Xu2005}. In \cite{Burgholzer2007b}, a far-field UBP formula was derived to reconstruct two-dimensional initial sources from three-dimensional data measured by integrating line detectors arranged on a rotating cylindrical detection surface. In \cite{Haltmeier2014}, UBP-type reconstruction formulas were proposed for recovering a function in arbitrary dimensions from its spherical means centered on the boundary of a convex domain with arbitrary shape.

In \cite{Ammari2010}, two methods were proposed for reconstructing small absorbers from the induced acoustic waves measured on the boundary. These approaches are based on isolating the photoacoustic signal generated by a targeted optical absorber from those produced by other absorbers in the medium \cite{Ammari2010}. The corresponding reconstruction formulas were derived for both Neumann and Dirichlet boundary conditions imposed on the boundary.

It was shown in \cite{Hristova2008} that for any initial pressure source with bounded support, the resulting wave leaves any bounded domain within a finite time \(T\). Consequently, the initial source can be reconstructed by reversing the wave equation in time, using a vanishing Cauchy condition at the final time \(T\), together with a boundary condition given by a time-reversed version of the data measured on the boundary over the time interval \([0,T]\) \cite{Hristova2008}. It was further shown that the time-reversal (TR) approach can exactly reconstruct the initial source only in the case of constant sound speed and in odd spatial dimensions \cite{Finch2004}. Nevertheless, the TR approach remains sufficiently accurate for variable non-trapping sound speeds \cite{Hristova2008}, particularly when the measurement data are available over the entire boundary. In \cite{Burgholzer2007a}, a numerical implementation of a second-order wave equation system approximating the time-reversal approach on a Cartesian grid was proposed. In \cite{Nguyen2016}, a TR approach was developed for the case of a perfectly reflecting boundary. The associated second-order wave system employs mixed boundary conditions, which are implemented using finite differences in time on a Cartesian grid.

An iterative implementation based on a forward operator and a modified time-reversal operator, formulated as a convergent Neumann series algorithm, was proposed in \cite{Stefanov2009} and \cite{Qian2011}. In this approach, the time-reversal procedure is modified such that the Cauchy condition at the final time \(T\) is given by the harmonic extension of the boundary data. The uniqueness and stability of the resulting Neumann-series inversion formula, under the assumption of variable and non-trapping sound speed, were established in \cite{Stefanov2009}. A numerical realization of the proposed Neumann series algorithm was presented in \cite{Qian2011}. The numerical results demonstrate that the algorithm converges even in the presence of trapping sound speeds, such as those encountered in brain imaging, or when the measurement data are only partially available on the boundary. In \cite{DeanBen2012}, an objective function describing the mismatch between the measured and modeled data was iteratively minimized. Under the assumption of constant sound speed, this minimization approach employed the Poisson-type integral proposed in \cite{Kruger1995} to iteratively model the data on the boundary.

In the context of iterative minimization of an objective function, the adjoint of the photoacoustic forward operator has been derived in several studies. In \cite{Arridge2016}, the photoacoustic operator is modeled as a mapping from the initial source to pressure fields averaged over small but finite spatial volumes whose centers lie on the boundary \cite{Arridge2016}. The resulting adjoint operator resembles the spherical-mean backprojection operator proposed in \cite{Kruger1995}, but with the integration performed over small but finite spatial volumes centered on the boundary, rather than over the boundary surface itself.  An adjoint formulation equivalent to that presented in \cite{Arridge2016} was derived directly on a discretized domain in \cite{Huang2013}. The work in \cite{Huang2013} follows a discretize-then-adjoint approach, whereas \cite{Arridge2016} adopts an adjoint-then-discretize strategy. The latter was later generalized to viscoelastic media \cite{Javaherian2018}.

In \cite{Belhachmi2016}, a hybrid approach combining the Finite Element Method (FEM) and the Boundary Element Method (BEM) was proposed to approximate the forward and adjoint operators of photoacoustic tomography in the time domain. This framework is based on weak formulations of the forward and adjoint problems with Neumann boundary conditions, where the boundary conditions at each time step are obtained through a boundary element formulation. In \cite{Haltmeier2017}, using the same weak formulation, two systems of adjoint wave equations were derived that solve the associated time-reversed single-layer potential integral formulations. In \cite{Do2018}, reconstruction formulas for recovering the standard Radon projections of the initial source from boundary data were extended to the case where the data are available only over a reduced time interval and on a truncated acquisition surface. These reconstruction formulas rely on single-layer potential integral representations to reconstruct the initial source from the measured boundary data. In \cite{Kuchment2025}, necessary and sufficient conditions were established for reconstructing the initial source from spherical means centered on the boundary when the boundary data are available only over a half-time interval. This analysis is based on a spherical-harmonic representation of the Kirchhoff--Helmholtz integral formula evaluated over the boundary and the measurement time interval.

All of the approaches discussed so far rely on a smoothness assumption on the boundary. This assumption is either applied to both the reception step of the forward operator and, equivalently, the emission step of the backprojection (or adjoint) operator, or it is imposed solely on the reception step of the forward operator, in which case the resulting discontinuity is treated as a single-layer potential in the formulation of the backprojection (or adjoint) operator. Accordingly, in defining the reception step, all these studies assume that a continuous field is extracted on the boundary via sampling. Recall that sampling on the boundary in a continuous domain is equivalent to interpolation from discrete sampled points onto the boundary in a discretized domain \cite{Huang2013}.

\subsection{Our Contributions}
This study extends our recently proposed forward--adjoint pair of operators for solving inverse problems in acoustics to the specific application of photoacoustic tomography \cite{Javaherian2022}. While the forward–adjoint operator proposed in \cite{Javaherian2022} is formulated in a general framework for acoustic inverse problems and accommodate both Neumann and Dirichlet boundary conditions, involving both emission and reception processes, the particular focus of this manuscript is on photoacoustic tomography.

Accordingly, the acoustic step of the forward problem in photoacoustic tomography is formulated as an operator that maps the initial source, whose support lies within a volumetric open spatial domain, to measured Dirichlet-type data on its boundary over time. The key component of the acoustic forward operator proposed in \cite{Javaherian2022} is a time-independent reception map from free space to the boundary. In \cite{Javaherian2022}, the derivation of this reception operator is presented in detail. It is shown that, at each arbitrary but fixed time, the reception operator conceptually acts as a left inverse of an emission operator that maps Dirichlet-type source data prescribed on the boundary into free space. In a distributional sense, the emission operator treats a boundary-restricted source field as a smeared singularity in its normal derivative across the boundary in free space. Conversely, the reception operator extracts the data restricted to the boundary as a smeared singularity in its normal derivative across the boundary in free space.  Defining the photoacoustic forward operator through this reception map leads to an adjoint operator equivalent to a time-reversed double-layer potential integral formula, yielding a formulation that is fully consistent with the fundamental principle of reciprocity in physics.

\section{The PAT Forward and Adjoint Operators}   \label{sec:pat-forward-adjoint}

This section presents the derivation of the mathematical formulations for the photoacoustic forward operator and its adjoint. Let $\Omega \subset \mathbb{R}^d$ be an open domain bounded by the observation boundary $\partial \Omega$, where the spatial dimension is denoted by $d \in \{2,3\}$. Accordingly, Cartesian coordinates are indexed by $\zeta \in \{1,\ldots,d\}$. Also, let $c(\bx)$ and $\rho_0(\bx)$ denote the sound speed and ambient density maps, respectively.

Short laser pulses are absorbed in a semitransparent tissue and generate an initial pressure distribution $p_0 \in C_0^\infty(\Omega)$, which is extended by zero to $\mathbb{R}^d$. Each receiver is indexed by $r$ and is defined as an infinite plane $\partial \nu_r$ that divides the full space into two half-spaces $\nu_r^{\pm}$. The receiver support is restricted to a finite surface lying on $\partial \Omega$, with $\Omega \subset \nu_r^{+}$.

Accordingly, the Dirichlet-type measured data associated with receiver \( r \) is denoted by \( y_r^D \in C_0^{\infty} \big( \partial \Omega \times [0,T] \big) \), and is defined as
    \begin{align} \label{eq:dirichlet-data}
        y_r^D(\bx_s, t) := 
        \begin{cases}
            \displaystyle p(\bx_s, t), & \text{if }\bx_s \in \partial \nu_r \cap \partial \Omega , \\[0.5em]
            0, & \text{otherwise},
        \end{cases}
    \end{align}
whose support is restricted to $\partial \nu_r \cap \partial \Omega$ and is extended by zero to the entirety of $\partial \Omega$. Note that we assume $y_r^D$ varies spatially across the receiver surface. In practice, such spatial variation can arise from the apodization or spatially varying sensitivity of the receiver aperture.

Below, we shall use the regularized Dirac distribution \(\delta_b\), which is defined by
\begin{align} \label{eq:delta-regularized}
\delta_b(\bx - \bx') 
= \prod_{\zeta=1}^d \frac{1}{b} \;
\mathrm{sinc}\!\left( \frac{\pi (x^\zeta - x'^\zeta)}{b} \right),
\end{align}
where \(b>0\) is a bandwidth parameter controlling the spread of the approximation (units: \(\mathrm{m}^{-d}\)). 
The sinc function is defined by
\[
\mathrm{sinc}(f) = \frac{\sin(f)}{f}, 
\qquad \mathrm{sinc}(0) = 1.
\]

\subsection{PAT Forward Operator}  \label{sec:pat-forward}

Consequently, the acoustic forward operator $\mathcal{A}$ is the map:
\begin{align}  \label{eq:operator-forward}
\begin{split}
\mathcal{A}: C_0^{\infty}(\mathbb{R}^d \times \mathbb{R}) &\rightarrow  C_0^{\infty}\left(\partial \Omega \times [0,T]\right) , \\
\mathcal{A}[s_{\mathcal{Q}}] &= y^D,
\end{split}
\end{align}
satisfying the wave system \cite{Javaherian2022}
\begin{align} \label{eq:wave-pat}
\begin{split}
 &\left[ \frac{1}{c^2}\frac{\partial^2}{\partial t^2} - \nabla \cdot \left( \frac{1}{\rho_0} \nabla \right)  \right] p(\bx, t) =  s_{\mathcal{Q}},  \quad (\bx,t) \in \mathbb{R}^d \times [0,T], \\
&y_r^D(\bx_s, t) = \mathcal{R}_{\mathcal{G},r}^D \big[p (\bx,t) \big], \quad  (\bx_s,t) \in \partial \Omega \times [0,T],
\end{split}
\end{align}
with initial conditions
\begin{align}  \label{eq:cauchy1}
    p(\bx, t) \big|_{t=0} = 0, \quad 
    \frac{\partial p}{\partial t} (\bx, t) \big|_{t=0} = 0.
\end{align}

Here, we have used the \emph{reception operator} \(\mathcal{R}_{\mathcal{G},r}^D\) associated with receiver \(r\), satisfying \cite{Javaherian2022}
\begin{align} \label{eq:reception}
\begin{split}
\mathcal{R}_{\mathcal{G},r}^D : C_0^{\infty}\!\left( \mathbb{R}^d \times [0,T] \right)
&\longrightarrow 
C_0^{\infty}\!\left( \partial \Omega \times [0,T] \right), \\
\lim_{b \to 0^+} \, \frac{2}{a_R(\bx_s)}\, f\big|_{\partial \nu_r} (\bx_s,t)
&=
\int_{\mathbb{R}^d}  \, d\bx  \, \delta_b(\bx_s - \bx)
\, \frac{\partial f}{\partial \bn_{r,+}}(\bx,t) 
\quad (\bx_s,t)  \in (\partial \nu_r \cap \partial \Omega) \times [0,T],
\end{split}
\end{align}
where \( f\big|_{\partial \nu_r}\) denotes the restriction of the field \(f\) to the infinite plane \(\partial \nu_r\), and \( \partial / \partial \bn_{r,+} \) denotes the outward normal derivative with respect to the half-space associated with the receiver, i.e., the normal vector pointing from \( \nu_r^{+} \) into \( \nu_r^{-} \).

The factor
\[
a_R(\bx_s) = \frac{d\bx_s}{dS(\bx_s)}
\]
is a geometric scaling term accounting for the relative infinitesimal size (Jacobian factor) of the free-space volumetric element with respect to the surface element on \(\partial\nu_r\).

The reception operator extracts the acoustic field on the receiver surface as a smeared singularity of its outward normal derivative across the receiver surface in free space, interpreted in a distributional sense \cite{Javaherian2022}.

The formulation above provides the most general definition of the operator. As a concrete example, we focus our study on the forward operator associated with acoustic wave propagation in photoacoustic tomography.
In this setting, the overall forward map is realized as the composition \( \mathcal{A}_r \circ \mathcal{Q} \). Here, 
\( \mathcal{Q} \) denotes the operator that converts the locally induced 
instantaneous pressure distribution \( p_0 \in C_0^\infty(\Omega) \) into a temporaly-regularized source \( s_{\mathcal{Q}} \in C_0^\infty(\mathbb{R}^d \times \mathbb{R}) \), i.e.,
\begin{align} \label{eq:photoacoustic-op}
    \begin{split}
        \mathcal{Q}: C_0^{\infty}(\Omega) &\rightarrow C_0^{\infty}(\mathbb{R}^d \times \mathbb{R}), \\
        \mathcal{Q}[p_0] &= s_{\mathcal{Q}},
    \end{split}
\end{align}
where \(s_{\mathcal{Q}}\) denotes the temporally-regularized photoacoustic source, and is defined by
\begin{align}   \label{eq:photoacoustic-source}
    s_{\mathcal{Q}}(\bx, t) = \frac{1}{c(\bx)^2}\, p_0(\bx)\, \frac{d}{d t}\delta_q(t).
\end{align}
Here, \(\delta_q(t)\) denotes a regularized temporal Dirac delta distribution centered at the time origin, which decays away from \(t=0\). In this work, we use the following simplified approximation:
\begin{align} \label{eq:dirac-t}
\delta_q(t) =
\begin{cases}
\dfrac{1}{2q}, & \text{if } , \mid t \mid  \leq q, \\
0, & \text{otherwise},
\end{cases}
\end{align}
where \(q>0\) is a small time-width parameter controlling the width of the pulse. \footnote{Note that the physics of the photoacoustic phenomenon is originally formulated as an initial value problem for a homogeneous wave equation with a vanishing right-hand side and nonvanishing, heterogeneous Cauchy conditions. However, it can equivalently be reformulated as the inhomogeneous wave equation given in the first line of Eq.~\eqref{eq:wave-pat}, equipped with vanishing Cauchy conditions~\eqref{eq:cauchy1}.}

\subsection{Minimization Problem} \label{sec:pat-minimization}
To regularize the singularity arising in the reception operator, we consider a scaled positivity-constrained least-squares minimization problem of the form
\begin{align} \label{eq:objective-function}
\argmin_{p_0 \geq 0}  \quad  \chi [p_0]:\left\| \frac{1}{a_R} \mathcal{A} \circ \mathcal{Q} [p_0] - \frac{1}{a_R} \bigl(y^D\bigr)_{\mathrm{measured}} \right\|_{ L^2 \big( \partial \Omega \times [0,T]\big) }^2,
\end{align}
where  $\chi$ is the objective function, $ \frac{1}{a_R} y^D$ denotes the scaled Dirichlet-type measured data on the boundary, and \(\frac{1}{a_R} \mathcal{A}\) denotes the scaled version of the forward operator defined by Eqs.~\eqref{eq:operator-forward}–\eqref{eq:reception}, where the factor \(a_R\) appearing in the reception operator \eqref{eq:reception} is cancelled by the scaling factor \(\frac{1}{a_R}\). Additionally, \(\mathcal{Q}\) represents the photoacoustic source operator defined by Eqs.~\eqref{eq:photoacoustic-op}–\eqref{eq:dirac-t}.

Here, we minimize the associated scaled objective function iteratively using the projected gradient descent algorithm:
\begin{align} \label{eq:iterate}
p_0^{n+1} = \max\left(p_0^n + \tau d^n,\, 0\right),
\end{align}
where \(d^n\) is the step direction at iteration \(n\), given by
\begin{align} \label{eq:step-direction}
d^n = - 2 \mathcal{Q}^* \circ \left(\frac{1}{a_R}\mathcal{A}\right)^* 
\left[
\frac{1}{a_R}\mathcal{A} \circ \mathcal{Q}  \mathcal[p_0^n]
-
\frac{1}{a_R}\bigl(y^D\bigr)_{\mathrm{measured}}
\right],
\end{align}
and scaled by the fixed step length \(\tau\).

The minimization is performed iteratively using the updates in~\eqref{eq:iterate} until the iteration index \(n\) for which
\begin{align} \label{eq:termination}
    \frac{\left\|p_0^{n} - p_0^{n-1}\right\|_2}{\|p_0^{\,n-1}\|_2} < \eta,
\end{align}
is satisfied, where \(  \eta \) is a small tolerance parameter. The step directions given by Eq.~\eqref{eq:iterate} depend on the full composite PAT forward operator and its adjoint formulation. Having described the minimization procedure, we derive the full PAT adjoint operator in the next section.

\subsection{PAT Adjoint Operator} \label{sec:pat-adjoint}

In Sections~\ref{sec:pat-forward} and~\ref{sec:pat-minimization}, we defined the photoacoustic forward operator as the scaled composite map \( \frac{1}{a_R}\mathcal{A} \circ \mathcal{Q} \), acting on the initial pressure distribution \( p_0 \in C_0^\infty(\Omega) \) to produce the scaled measured Dirichlet-type boundary data, i.e., \( \frac{1}{a_R} y^D \in C_0^\infty(\partial \Omega \times [0,T]) \). Correspondingly, the adjoint of the photoacoustic forward operator is given by the composite map \( \mathcal{Q}^* \circ \left( \frac{1}{a_R} \mathcal{A} \right)^* \), acting on the scaled Dirichlet-type
boundary data \( \frac{1}{a_R} y_r^D \). Accordingly, the adjoint of the photoacoustic source operator \(\mathcal{Q}\) is derived in the following lemma.

\begin{lemma}
 The action of the adjoint of the operator \(\mathcal{Q} \), defined by Eq.~\eqref{eq:photoacoustic-op}, on any test function 
    \( f^{s_{\mathcal{Q}}} \in C_0^\infty(\mathbb{R}^d \times \mathbb{R}) \) is given by
    \begin{align}  \label{eq:adjoint-photoacoustic}
        \begin{split}
        &\mathcal{Q}^*: C_0^\infty(\mathbb{R}^d \times \mathbb{R}) 
        \rightarrow C_0^\infty(\Omega), \\
        &\mathcal{Q}^*[f^{s_{\mathcal{Q}}}]
        = - \frac{1}{c(\bx)^2}\,  \,   \delta_q(t) \, \frac{\partial  f^{s_{\mathcal{Q}} } }{\partial t}
       (\bx,t) \rvert_{t=0}.
       \end{split}
    \end{align}
\end{lemma}

\begin{proof}
The adjoint operator \( \mathcal{Q}\) and its adjoint, with respect to the standard 
\(L^2\) bilinear form in  \(C_0^\infty(\Omega)\) and \(C_0^\infty(\mathbb{R}^d \times \mathbb{R})\), should satisfy
\begin{align} \label{eq:photoacoustic-adj-1}
        \int_{\mathbb{R}^d} d\bx \, \int_{\mathbb{R}} dt \,
        f^{s_{\mathcal{Q}}}(\bx,t)  \,
        \frac{1}{c(\bx)^2} f_0(\bx)  \,  \frac{d}{dt}  \delta_q(t)
        = \int_{\Omega} d\bx \, \int_{\mathbb{R}} dt \, \mathcal{Q}^*[f^{s_{\mathcal{Q}}}]  \, f_0(\bx)
\end{align}
for any  \( f_0 \in C_0^\infty(\Omega)\)~and~\(f^{s_{\mathcal{Q}}} \in C_0^\infty(\mathbb{R}^d \times \mathbb{R})\).

Applying integration by parts in time to the left-hand side of 
\eqref{eq:photoacoustic-adj-1} yields
\begin{align} \label{eq:photoacoustic-adj-2}
    \int_{\Omega} \, d\bx \,   \int_{\mathbb{R}}   dt  \,  \Big[ - \frac{1}{ c(\bx)^2} 
    \delta_q(t) \
    \frac{\partial f^{s_{\mathcal{Q}}} }{\partial t} (\bx,t)  \Big]
    f(\bx),
\end{align}
where we have also restricted the domain of spatial integration to the support of \(f\).

Comparing the integral expression in \eqref{eq:photoacoustic-adj-2} with the right-hand side of Eq.~\eqref{eq:photoacoustic-adj-1} immediately identifies the bracketed term in \eqref{eq:photoacoustic-adj-2} as \(\mathcal{Q}^*[f^{s_{\mathcal{Q}}}]\), which proves the expression in \eqref{eq:adjoint-photoacoustic}.
\end{proof}

Additionally, the adjoint of the scaled acoustic forward operator \( \frac{1}{a_R} \mathcal{A} \) was derived in \cite{Javaherian2022}, and we refer the reader to that study for further details. Correspondingly, the action of the adjoint of the scaled forward operator \( \frac{1}{a_R} \mathcal{A} \), introduced in Eq.~\eqref{eq:operator-forward}-\eqref{eq:reception}, on any scaled Dirichlet-type boundary data \(\frac{1}{a_R} f^D \in C_0^{\infty}(\partial \Omega \times [0,T])\) is given by the adjoint wave equation:

\begin{align}  \label{eq:wave-adjoint}
\left[
\frac{1}{c^2}\frac{\partial^2}{\partial t'^2}
-
\rho_0 \nabla_{\bx'} \cdot
\left(
\frac{1}{\rho_0}\nabla_{\bx'}
\right)
\right]
p^*(\bx',t')
=
\frac{1}{2}  \sum_r
\int_{\partial \nu_r \cap \partial \Omega}
\, dS(\bx_s)  \,
\nabla \delta_b(\bx'-\bx_s)
\cdot
\left[
f_r^D(\bx_s,T-t')\,\bn_{r,-}
\right]
\end{align}
subject to the Cauchy conditions
\begin{align}   \label{eq:cauchy-adjoint}
p^*(\bx',0)=0,
\qquad
\frac{\partial p^*}{\partial t'}(\bx',0)=0,
\qquad
\bx' \in \mathbb{R}^d .
\end{align}
Here, \(p^*\) denotes the free space--time adjoint wavefield, i.e., the solution of the adjoint wave equation, and \(f_r^D\) denotes the component of the Dirichlet-type boundary data \(f^D\) whose support is confined to \( \partial \nu_r \cap \partial \Omega \), and which is extended by zero to \(\partial \Omega\) (cf. Eq.~\eqref{eq:dirichlet-data}). Moreover, \(\bn_{r,-}\) denotes the inward-directed normal vector to the receiver surface \(r\). (The reader is referred to \cite{Javaherian2022} for the proof.)

Correspondingly, the full composite PAT adjoint operator is defined as the map
\begin{align}  \label{eq:full-operator-adjoint}
\begin{split}
\mathcal{Q}^* \circ \left(\frac{1}{a_R}\mathcal{A}\right)^*         & :  C_0^{\infty}\left(\partial \Omega \times [0,T]\right)  \rightarrow  C_0^{\infty}(\Omega),\\
\mathcal{Q}^* \circ \left(\frac{1}{a_R}\mathcal{A}\right)^* [\frac{1}{a_R} f^D]& : =   \frac{1}{c(\bx')^2}  \, \delta_q(t') \ \frac{\partial p_r^* }{\partial t}(\bx', T-t') \rvert_{t'=0},
\end{split}
\end{align}
where \(p^*\) satisfies the adjoint wave equation~\eqref{eq:wave-adjoint}, together with the Cauchy conditions~\eqref{eq:cauchy-adjoint}.

Consequently, the full composite forward and adjoint operators should satisfy
\begin{align} \label{eq:inner-product-test}
\int_{\partial \Omega} \, d \bx_s \,  \int_0^T  \, dt \, \frac{1}{a_R} y^D(\bx_s,t) \left(  \frac{1}{a_R} \mathcal{A} \right) \circ \mathcal{Q} [p_0]  = \int_{\Omega} \, d \bx   \,  \int_{\mathbb{R}} \,  dt \ \mathcal{Q}^*\circ  \left(  \frac{1}{a_R} \mathcal{A} \right)^* \big[ \frac{1}{a_R} y^D\big] p_0(\bx),
\end{align}
for any \( p_0 \in c_0^\infty\left(\Omega\right) \) and \( \frac{1}{a_R} y^D \in C_0^\infty \left(\partial \Omega \times [0,T]\right) \).

\subsection{System of Three-Coupled First-Order Wave Equations}

We employ a three-coupled first-order linear formulation of the forward and adjoint operators defined above. The corresponding wave system is defined in terms of the vector‑valued particle velocity field \(\bu\) and the scalar‑valued acoustic density \(\rho\) and pressure \(p\) fields. Accordingly, the free-space wave equation in the first line of Eq.~\eqref{eq:wave-pat} is reformulated as the following first-order wave system:
\begin{align}
\label{eq:wave-system}
\begin{split}
    & \frac{\partial}{\partial t} \bu(\bx, t) = -\frac{1}{\rho_0(\bx)} \nabla p(\bx, t) + \boldsymbol{\mathcal{S}}_f(\bx, t), \\
    & \frac{\partial}{\partial t} \rho(\bx, t) = -\rho_0(\bx) \nabla \cdot \bu(\bx, t) + s_m(\bx, t), \\
    & p(\bx, t) = c(\bx)^2 \rho(\bx, t),
\end{split}
\end{align}
where \(s_m\) denotes a scalar mass source and \(\boldsymbol{\mathcal{S}}_f\) denotes a vector-valued force source that is regularized in space, as detailed below.

\subsubsection{Scaled Full Composite PAT Forward Operator}

Correspondingly, the forward operator defined in Section~\ref{sec:pat-forward} is obtained by setting\cite{Javaherian2022}
\begin{align}  \label{eq:mass-source}
    s_m(\bx, t)
    = \int_{-\infty}^t dt'\, s_{\mathcal{Q}}(\bx, t')
    = \frac{1}{c(\bx)^2}\,p_0(\bx)\,\delta_q(t),
    \qquad
    \boldsymbol{\mathcal{S}}_f = 0.
\end{align}

The wave system in Eq.~\eqref{eq:wave-system} yields the free-space pressure wavefield, which is subsequently mapped onto each receiver surface \(r\) using the reception operator \(\mathcal{R}_{\mathcal{G},r}^D\), introduced in Eq.~\eqref{eq:reception}, to obtain the scaled Dirichlet-type boundary data \(\frac{1}{a_R}y_r^D\).

\subsubsection{Scaled Full Composite PAT Adjoint Operator}

Similarly, the scaled adjoint operator defined in Section~\ref{sec:pat-adjoint} is obtained by the settings \cite{Javaherian2022}
\begin{align}
\boldsymbol{\mathcal{S}}_f(\bx, t)
=
- \frac{1}{2 \rho_0(\bx)} \, \sum_r
\int_{\partial\nu_r \cap \partial\Omega}
dS(\bx_s)\,
\delta_b(\bx - \bx_s)\,
\left[p(\bx_s, T - t)\,\bn_{r,-} \right],
\qquad
s_m = 0.
\end{align}

The output of the full PAT adjoint operator is then obtained from the time-reversed free-space pressure wavefield through the second line of Eq.~\eqref{eq:full-operator-adjoint}.

\subsection{Discretized Algorithm}

This section outlines a discretization of the forward and adjoint operators defined above. We use a regular grid with sampled points \( \mathbf{X} = \{X^\zeta : \zeta \in \{1,\ldots,d\}\} \), where \( \zeta \) indexes the Cartesian coordinates. A uniform grid spacing \( \Delta x \) is used for all Cartesian coordinates \( \zeta \), without loss of generality. Each sampled point is indexed by \( i \in \{1,\ldots,N_i\} \). Furthermore, let \( \bt \in \{0,\ldots,N_t\} \) denote the discrete time indices corresponding to the measurement period \( t \in [0,T] \). The time spacing is denoted by \(\Delta t\).

Each receiver surface is discretized using a triangulation and divided into finite elements \(K \in \{1,\ldots,N_K\}\), which are triangles for \(d=3\) and line segments for \(d=2\), with area or length \(s_K\), respectively. Here, \(N_K\) denotes the total number of finite elements across all disjoint receivers. The nodes are located at points \(\bx_j\), with \(j \in \{1,\ldots,N_j\}\), where \(N_j\) is the total number of nodes across all disjoint receivers. Each finite element \(K\) is associated with a local set of node indices \( l(K) \subseteq \{1,\ldots,N_j\} \). The set \( l(K) \) identifies the \( N_{l(K)} \) nodes that define the vertices of element \(K\) (\(N_{l(K)} = d\)). These nodes are used for interpolation and quadrature on the finite element \(K\). Accordingly, each receiver \(r\) is discretized into \(N_{j,r}\) nodes, so that the total number of nodes across all receivers is \(N_j = \sum_r N_{j,r}\). Likewise, each receiver \(r\) is discretized into \(N_{K,r}\) finite elements, giving a total of \(N_K = \sum_r N_{K,r}\) finite elements across all receivers.

In the numerical experiments presented in this manuscript, we assume \( d = 2 \). For this setting, each receiver is modeled as a line segment composed of finite elements
corresponding to connected line subsegments, such that \( N_{K,r} = N_{j,r} - 1 \). Additionally, for simplicity, we assume that \(N_{j,r}\) and \(N_{K,r}\) are fixed for all receiver indices \(r\).

Bar notations are used to denote fields defined on the fully discretized domain. The discretizations of a scaled form of the PAT forward operator introduced in Section~\ref{sec:pat-forward}, acting on the initial pressure distribution \(p_0\), and the scaled PAT adjoint operator introduced in Section~\ref{sec:pat-adjoint}, acting on the scaled Dirichlet-type boundary data \(\frac{1}{a_R} y^D\), are outlined in Algorithm~\ref{alg:1}. 

In this algorithm, the discretization is performed on a grid staggered in both space and time. The directional spatial derivatives are approximated using a k-space pseudospectral method \cite{Tabei2002,Treeby2010}. While the wave system given in Eq.~\eqref{eq:wave-pat} is defined in free space, the discretized computational grid has a finite size. To model the propagation of the pressure wavefield in free space, absorbing boundary conditions of the Perfectly Matched Layer (PML) type are employed to prevent waves leaving one side of the computational grid from reappearing on the opposite side. Accordingly, we use a direction-dependent Perfectly Matched Layer (PML) operator \(\Lambda^{\zeta}\) satisfying
\begin{align}
\Lambda^{\zeta} = e^{-\alpha^{\zeta} \Delta t / 2},
\end{align}
where \(\alpha^{\zeta}\) is the virtual absorption coefficient of the PML layer along the Cartesian coordinate \(\zeta\).

\begin{algorithm}
    \caption{Full-discretization at time step \( \bt \in \{-1, \ldots, N_t-1\} \)}
    \label{alg:1}
    \begin{algorithmic}[1]
        \State \textbf{Input:} \( \bar{c}, \bar{\rho}_0, \Delta t, \Lambda^\zeta, \bar{\mathcal{S}}_f^{\zeta}(\bX, \bt) \ (\zeta \in \{1, \ldots, d\}), \bar{\mathcal{S}}_m(\bX, \bt+\frac{1}{2})  \)
        \State \textbf{Initialize:} \( \bar{p}(\bX, -1) = 0, \ \bar{\bro}(\bX, -1) = 0, \ \bar{\bu}(\bX, -\frac{3}{2}) = 0 \) \Comment{Set Cauchy conditions} 
        \For{\( \bt = -1, \ldots, N_t-1 \)}
            \State \( \bar{u}^\zeta (\bX,\bt + \frac{1}{2})  \gets \Lambda^\zeta \ \Big[\Lambda^\zeta \bar{u}^\zeta(\bX, \bt -\frac{1}{2}) - \Delta t \frac{1}{\bar{\rho}_0(\bX)} \frac{\partial}{\partial \zeta} \bar{p}(\bX, \bt)\Big] 
            + \Delta t \bar{\mathcal{S}}_f^{\zeta} (\bX, \bt) \) \Comment{Update \( \bu \)}
            \State \( \bar{\rho}^\zeta(\bX, \bt+1)  \gets \Lambda^\zeta \Big[\Lambda^\zeta \bar{\rho}^\zeta(\bX, \bt) - \Delta t \bar{\rho}_0(\bX) \frac{\partial}{\partial \zeta} \bar{u}^\zeta(\bX, \bt+\frac{1}{2})\Big] + \Delta t \bar{\mathcal{S}}_m^\zeta (\bX, \bt+\frac{1}{2}) \) \Comment{Update \( \bro \)}
            \State \( \bar{p}(\bX, \bt+1) \gets \bar{c}(\bX)^2 \sum_{\zeta = 1}^d \bar{\rho}^\zeta (\bX, \bt+1) \) \Comment{Update \( p \)}
        \EndFor
    \end{algorithmic}
\end{algorithm}

\subsubsection{Discretized Scaled PAT Forward Operator}  \label{sec:forward-dis}

The discretized form of the PAT forward operator introduced in Section~\ref{sec:pat-forward} is implemented using Algorithm~\ref{alg:1} and employs, for all Cartesian coordinates \( \zeta \), the following sources:
\begin{align}
\bar{\mathcal{S}}_f^{\zeta} =0, \quad \bar{\mathcal{S}}_m^\zeta \left(\bX, \bt+\frac{1}{2} \right):=\frac{1}{2 \bar{c}(\bX)^2 d \Delta t}
\begin{split}
\begin{cases}
\mathfrak{S}\bar{p_0},  & \mathfrak{t} \in \{-1,0\} \\
0, & \text{Otherwise},
\end{cases}
\end{split}
\end{align}
which is a discretization of the mass source defined by Eq.~\eqref{eq:mass-source}, with the setting \(q = \Delta t\). Here, \( \mathfrak{S} \) is a self-adjoint smoothing operator enforced on the initial pressure distribution \( \bar{p}_0 \) to mitigate unintended oscillations in the numerical solution arising from the high spatial frequencies of  \( \bar{p}_0 \) \cite{Arridge2016}. Note that while the discretized acoustic density field $\bar{\rho}$ and the discretized mass source $\bar{\mathcal{S}}_m$ are scalars, they are implicitly divided by the Cartesian coordinates to accommodate the direction-dependency of the PMLs.

The reception step of the discretized scaled forward operator is then given by
 \begin{align}  \label{eq:forward-reception-dis}
\frac{1}{a_R} \bar{y}^D\left(\bx_j, \mathfrak{t}\right) := \frac{1}{a_R} \bar{\mathcal{A}} \circ \bar{\mathcal{Q}} [\bar{p_0}] (\bx_j, \mathfrak{t} ) \approx \frac{1}{2} \sum_{i,\bar{\delta}_b(\bX_i- \bx_j) > \epsilon} \bar{\delta_b}\left( \bx_j -\bX_i  \right) \bar{\frac{\partial p}{\partial \bn_{r,+}}} \left( \bX_i , \bt \right)    \quad \text{for all} \quad  \bt \in \{0, \ldots, N_t\},
\end{align} 
which is computed and aggregated for all receiver indices \(r\). Recall that the bar notation in \( \bar{\frac{\partial p}{\partial \bn_{r,+}}} \) denotes the discretization of the pressure derivative in the outward normal direction to the surface of receiver \(r\). Note that the values of \(\bar{\delta}_b\) smaller than a prescribed threshold \(\epsilon\) have been set to zero so that the regularized Dirac distribution is confined to its effective support, thereby reducing the computational cost. Additionally, the setting \(b = \Delta x\) for the Dirac delta regularization is a standard practice that aligns the regularization kernel with the underlying grid resolution.

\subsubsection{Discretized Scaled PAT Adjoint Operator}  \label{sec:adjoint-dis}

The discretized form of the scaled PAT adjoint operator introduced in Section~\ref{sec:pat-adjoint} is implemented according to Algorithm~\ref{alg:1}, with an additional emission step:
\begin{align} \label{eq:adjoint-emission-dis}
\begin{split}
   \bar{\boldsymbol{\mathcal{S}}}_f (\bX_i, \bt) \approx -\frac{1}{2 \bar{\rho}_0(\bX_i)} \sum_{K=1}^{N_K} \frac{s_K}{N_{l(K)}}
   \sum_{\substack{ j \mid j \in l(K), \bar{\delta}_b(\bX_i- \bx_j) > \epsilon}}
   \bar{\delta}_b(\bX_i - \bx_j)  
   \left[\bar{p}(\bx_j, N_t- \bt-1) \bn_{r,-} \right], 
   \qquad 
   \bar{\mathcal{S}}_m^\zeta = 0,\\
   \text{for all} \quad  \bt \in \{-1, \ldots, N_t-1\},
\end{split}
\end{align}
where the contributions have been summed over all receiver indices \(r\). The output of Algorithm~\ref{alg:1} is then given by
\begin{align}  \label{eq:source-adjoint-dis}
\bar{\mathcal{Q}}^* \circ \left(\frac{1}{a_R} \bar{\mathcal{A}}\right)^* [ \frac{1}{a_R} \bar{y}^D](\bX)
= \mathfrak{S} \,  \frac{1}{2 \bar{c}(\bX)^2}  \,\frac{ \bar{p}(\bX,N_t) - \bar{p}(\bX,N_t-2)}{(\Delta t)^2},
\end{align}
where the smoothing operator \( \mathfrak{S} \) is self-adjoint.  Recall that we used \(b=\Delta x\)~and~\(q=\Delta t\) in the discretized formulas.

\section{Numerical Results}
This section presents numerical results demonstrating the accuracy of the derived forward--adjoint operator pair and their performance in a photoacoustic scenario. The simulations in free space were performed using the k-Wave toolbox, which employs a k-space pseudo-spectral method for numerical integration in space and time \cite{Treeby2010}, while the mappings between the free space and the boundary, in both directions, were carried out using the discretized formulas presented in Sections~\ref{sec:forward-dis} and~\ref{sec:adjoint-dis}.

\subsection{Verification of the Mathematical Consistency of the Forward--Adjoint Operator}

In this section, the accuracy of the derived PAT forward--adjoint operator pair is evaluated with respect the inner product indentity given in Eq.~\eqref{eq:inner-product-test}. Here, \( p_0 \in C_0^\infty(\Omega) \) and \( \frac{1}{a_R} y^D \in C_0^\infty(\partial\Omega \times [0,T]) \) are randomly generated. 

The accuracy of the derived adjoint operator is quantified by
\begin{align} \label{eq:inner-product-test-measure}
RD(\%) = \frac{\left\| LHS - RHS \right\|_2}{\left\| LHS \right\|_2} \times 100,
\end{align}
where \(LHS\) and \(RHS\) denote the left-hand side and right-hand side of the inner product test formula~\eqref{eq:inner-product-test}, respectively.

The computational grid consists of 296 sampled points along each Cartesian coordinate, of which 20 points on each side are assigned to PMLs. Accordingly, 256 sampled points lie within the domain \([-51.2, 50.8]\,\mathrm{mm}\) with a grid spacing of \(0.4\,\mathrm{mm}\) along each Cartesian coordinate. The pressure wavefield is simulated and recorded over the time interval \([0, 96.48]\,\mu\mathrm{sec}\) with a temporal spacing of \(0.08\,\mu\mathrm{sec}\), resulting in \(1207\) time steps. This setup corresponds to the parameter values \(d = 2\), \(N_i = 256^2\), \(\Delta x = 0.4\,\mathrm{mm}\), \(T = 96.48\,\mu\mathrm{sec}\), \(\Delta t = 0.08\,\mu\mathrm{sec}\), and \(N_t = 1207\).

The initial pressure source \(p_0\) is defined as a circular disc with a radius of \(36\,\mathrm{mm}\), with values ranging within \([0,1]\,\mathrm{Pa}\). We set up the following two experiments to quantify the consistency of the derived forward--adjoint PAT operators in satisfying the inner product test formula~\eqref{eq:inner-product-test}.

\subsubsection{On-Grid Line Receivers}

The receivers are modeled as line receivers aligned with grid sampling points on the left, bottom, and right sides of the domain, specifically at the second sampled points from each PML boundary. Each receiver is represented as a single line segment connecting two neighboring sampled points along the second row or column from the corresponding PML boundary. (We set $N_{j,r}=2,\; N_{K,r}=1$.) In this particular experiment, the receiver nodes coincide with the sampled grid points. Therefore, in the formulas \eqref{eq:forward-reception-dis} and \eqref{eq:adjoint-emission-dis}, used in the discretized forward and adjoint operators, respectively, the discretized form of the regularized Dirac distribution defined in Eq.~\eqref{eq:delta-regularized} is replaced by the exact Dirac distribution.

Accordingly, the initial pressure source \(p_0\) and the scaled Dirichlet-type boundary data \(\frac{y^D}{a_R}\) were randomly generated ten times in order to evaluate the accuracy of the inner product test formula~\eqref{eq:inner-product-test} using the quantification formula~\eqref{eq:inner-product-test-measure}. This resulted in an average RD value of \(9.32 \times 10^{-4}\%\), demonstrating the consistency of the derived forward--adjoint PAT operators in satisfying the inner product relation.

Figure~\ref{fig:1a} shows an example of a randomly generated initial pressure source \(p_0\). In this figure, the on-grid receivers on three sides of the computational grid are indicated in white. Figure~\ref{fig:1b} shows an example of the randomly generated, scaled Dirichlet-type boundary data $\frac{y^{D}}{a_{R}}$, where the horizontal and vertical axes correspond to time and receiver indices, respectively.

\begin{figure} 
\centering
\subfigure[]{\includegraphics[width=0.45\textwidth]{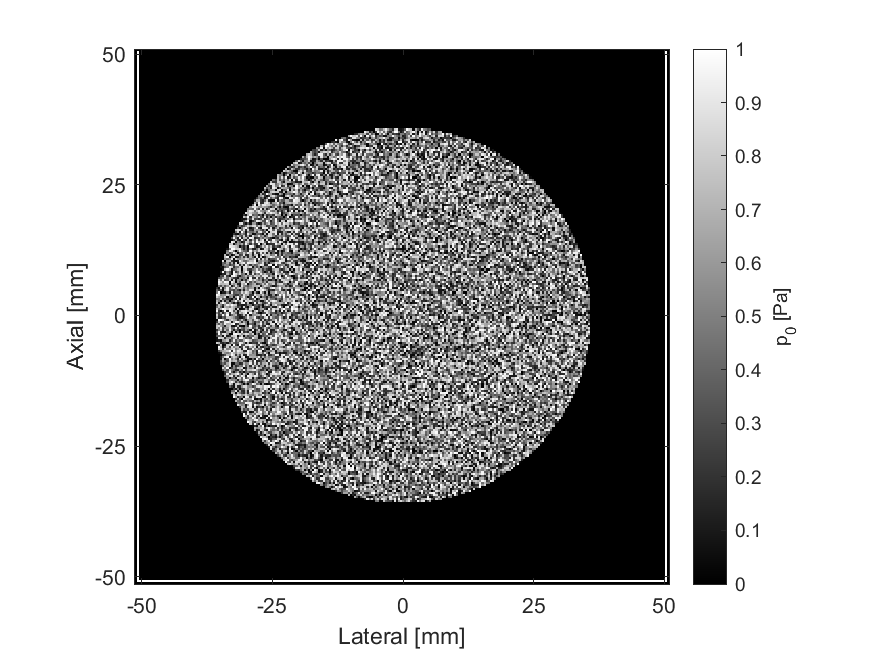}
\label{fig:1a}  }
\subfigure[]{\includegraphics[width=0.45\textwidth]{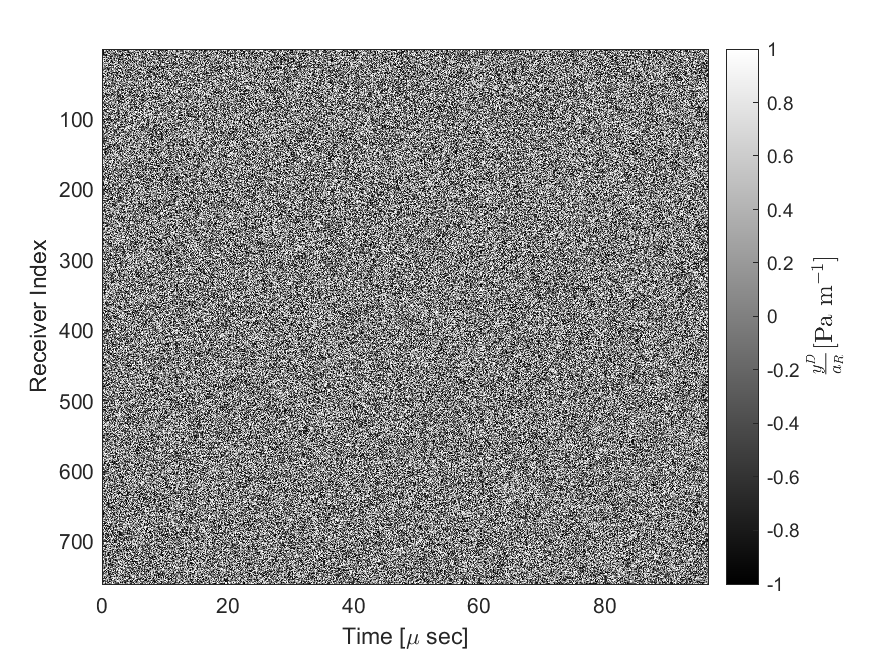}
\label{fig:1b} }
\caption{Experimental setup to quantify the consistency of the PAT forward--adjoint operator with respect to the inner product idenitity~\eqref{eq:inner-product-test} using on-grid line receivers. (a) Randomly generated initial pressure source distribution \(p_0\) and the on-grid line receivers on the left, bottom, and right sides of the grid. (b) Randomly generated scaled Dirichlet-type boundary data \(\frac{y^D}{a_R}\).} \label{fig:1}
\end{figure}

\subsubsection{Off-Grid Line Receivers}

A total of 64 receivers are modeled as line receivers with a half-length of \(2\,\mathrm{mm}\), whose centers are positioned on a circle with a radius of \(45\,\mathrm{mm}\). A dense sampling along each receiver is used by setting \((N_{j,r}=40,\; N_{K,r}=39)\). A 3D realization of this configuration may correspond to disc receivers with a radius of \(2\,\mathrm{mm}\), whose centers are positioned on a hemisphere with a radius of \(45\,\mathrm{mm}\). In the discretized formulas~\eqref{eq:forward-reception-dis} and~\eqref{eq:adjoint-emission-dis}, used in the forward and adjoint operators, respectively, a discretized form of the regularized Dirac distribution defined in Eq.~\eqref{eq:delta-regularized} is employed. To reduce the computational cost, values smaller than \(\frac{0.01}{b^d}\) (with \(d=2\)) are set to zero, thereby excluding sampled points that are far from the receiver nodes in the approximation. Figure~\ref{fig:2c} shows a binary mask indicating the sampled points that contribute to the approximation of the Dirac delta distributions associated with all receiver nodes.

Correspondingly, the initial pressure source \(p_0\) and the scaled Dirichlet-type boundary data \(\frac{y^D}{a_R}\) were randomly generated ten times to quantify the accuracy of the inner product test formula~\eqref{eq:inner-product-test} using the measure defined in~\eqref{eq:inner-product-test-measure}. This resulted in an average RD value of \(2.07 \times 10^{-5}\%\). The result demonstrates the consistency of the derived forward--adjoint PAT operators in satisfying the inner product formula.

Figure~\ref{fig:1a} shows an example of a randomly generated initial pressure source \(p_0\). The off-grid line receivers on the circular array are indicated in the figure. Figure~\ref{fig:1b} shows an example of the randomly generated scaled Dirichlet-type boundary data \(\frac{y^D}{a_R}\), where the horizontal and vertical axes correspond to time and receiver indices, respectively.

\begin{figure} 
\centering
\subfigure[]{\includegraphics[width=0.45\textwidth]{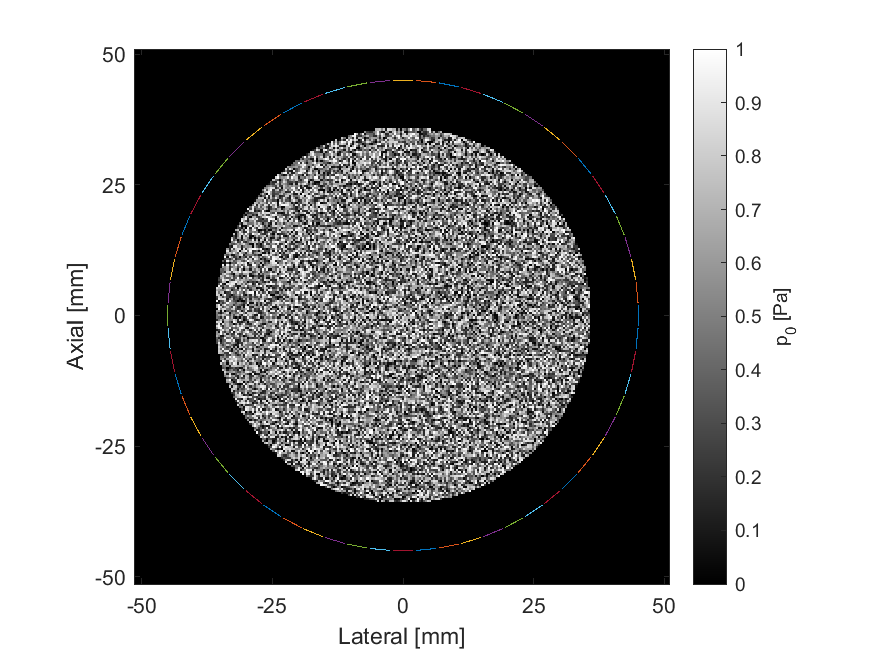}
\label{fig:2a}  }
\subfigure[]{\includegraphics[width=0.45\textwidth]{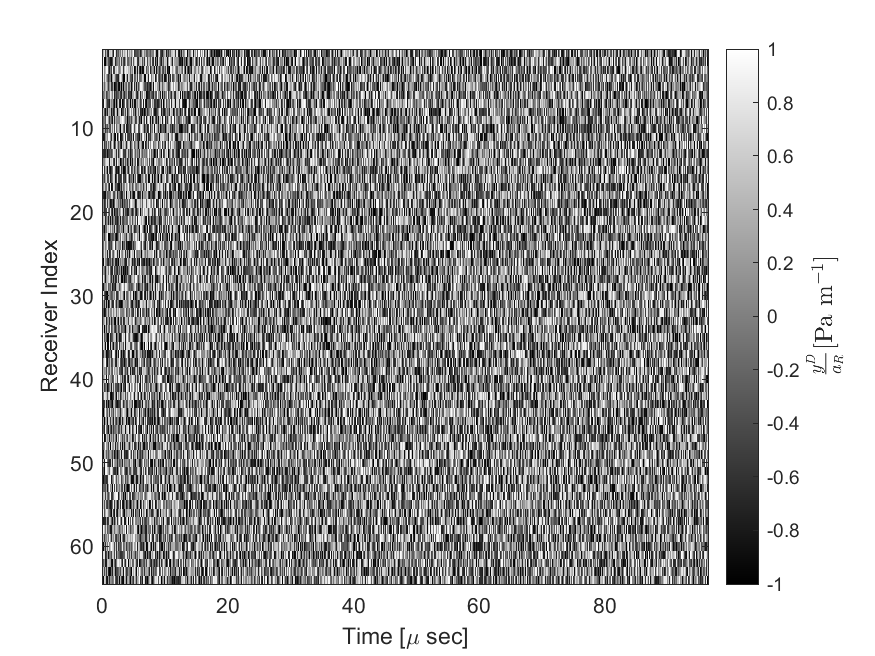}
\label{fig:2b} }
\subfigure[]{\includegraphics[width=0.45\textwidth]{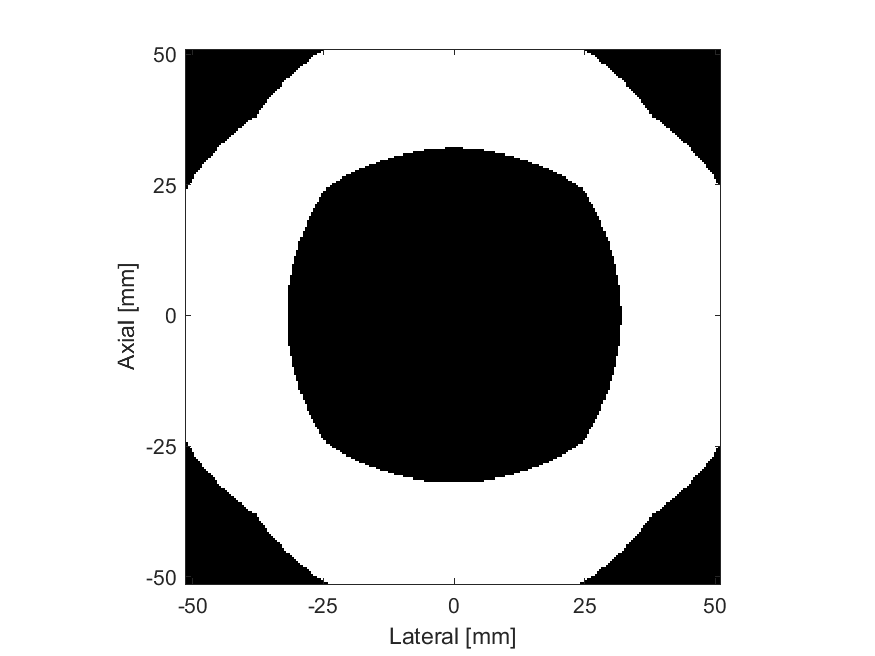}
\label{fig:2c} }
\caption{Experimental setup used to quantify the consistency of the PAT forward--adjoint operator with respect to the inner product identity~\eqref{eq:inner-product-test} using off-grid line receivers. (a) Randomly generated initial pressure distribution \(p_0\) and the circular array of off-grid line receivers. (b) Randomly generated scaled Dirichlet-type boundary data \(\frac{y^D}{a_R}\). c) Binary mask of the sampled points contributing to the approximation
of the Dirac delta distributions for all receiver nodes; black denotes false and white denotes true.}
  \label{fig:2}
\end{figure}

\subsection{Image Reconstruction}

This section presents a scenario for evaluating the performance of the iterative PAT image reconstruction algorithm introduced in Section~\ref{sec:pat-minimization}, which relies on the PAT forward and adjoint operators described in Sections~\ref{sec:pat-forward} and~\ref{sec:pat-adjoint}. To avoid an inverse crime, different discretization parameters are used for the simulation of the measured PAT data and for the iterative image reconstruction carried out on that data.

\subsubsection{Simulation of PAT Measurements}

The measured PAT data used as a benchmark for image reconstruction is simulated on a computational grid consisting of \(512\) sampled points, positioned within the range \([-11.0080, 10.9650] \, \mathrm{mm}\) with a grid spacing of \(43 \, \mu\mathrm{m}\) along both Cartesian coordinates. Perfectly matched layers (PMLs) consisting of \(20\) sampled points, using the same grid spacing as the main grid, are appended to all four sides. The pressure wavefield is simulated and recorded over the time interval \([0, 20.75]\, \mu\mathrm{sec}\) with a temporal spacing of \(5.733 \,  \mathrm{nsec}\), resulting in \(3621\) time steps. This setup corresponds to the parameter values \(d = 2\), \(N_i = 512^2\), \(\Delta x = 43\, \mu\mathrm{m}\), \(T = 20.75\, \mu\mathrm{sec}\), \(\Delta t = 5.733\, \mathrm{nsec}\), and \(N_t = 3621\). 

A total of \(256\) line receivers, each with a half-length of \(0.1\, \mathrm{mm}\), whose centers lie on a circle of radius \(9\, \mathrm{mm}\), are used to record the pressure wavefield in time using the discretized reception formula~\eqref{eq:forward-reception-dis}. Each receiver consists of \(19\) line segments connected by \(20\) nodes; that is, we use \(N_{j,r} = 20\) and \(N_{K,r} = 19\), which remain fixed for all receivers. In discretizing the reception step via formula~\ref{eq:forward-reception-dis}, we nulled the regularized Dirac delta distributions at sampled points far from the receiver nodes. Specifically, sampled points for which the value \(\bar{\delta}_b\) was less than \(\frac{0.001}{b^d}\) (with \(d = 2\)) were set to zero. Figure~\ref{fig:3a} shows a binary map indicating the sampled points that contribute to the approximation of the Dirac delta distributions across all receiver nodes. Figure~\ref{fig:3c} shows a vessel phantom used as the initial pressure source distribution \(p_0\) for this simulation. In this figure, the circular array of line receivers is shown in color.

After the simulation was performed, additive white Gaussian noise (AWGN) with a level of \(30\,\mathrm{dB}\) relative to the peak of each time trace was added to each time trace.

\subsubsection{PAT Image Reconstruction}

As explained in Section~\ref{sec:pat-minimization}, image reconstruction aims to determine the initial pressure source distribution \(p_0\) that minimizes the scaled objective function~\eqref{eq:objective-function}. This is accomplished through the iterative updates in~\eqref{eq:iterate}, where the corresponding step directions are computed using the formula~\eqref{eq:step-direction}. This procedure requires repeated evaluations of the PAT forward and adjoint operators. The step directions are scaled by a fixed step length. To avoid an inverse crime, numerical settings different from those used for simulating the PAT measurements are employed for reconstructing images from the measured PAT data.

The iterative computation of the forward and adjoint operators is performed on a computational grid consisting of \(440\) sampled points positioned within the interval \([-11.0000, 10.9500]\,\mathrm{mm}\), with a grid spacing of \(50\,\mu\mathrm{m}\) along both Cartesian coordinates. The computational grid is expanded on all four sides with Perfectly Matched Layers (PMLs) consisting of \(20\) sampled points, using the same grid spacing as the main grid. As the time array used for recording the boundary data is known in a practical setting, the pressure wavefield is simulated using the same time array as that used for the simulation of the PAT measurements. This setup corresponds to the parameter values \(d = 2\), \(N_i = 440^2\), \(\Delta x = 50\,\mu\mathrm{m}\), and the same time array parameters as those used for the simulation of the PAT measurements. 

The inverse crime is avoided not only in the discretization of the free-space domain (i.e., the positions of sampled points on the grid), but also in the discretization of the boundary (i.e., the positions of nodes on the receiver surfaces). Accordingly, for image reconstruction, each receiver consists of \(15\) line segments connected by \(16\) nodes; that is, we use \(N_{j,r} = 16\) and \(N_{K,r} = 15\), which remain fixed for all receivers. For the discretization of the reception step of the forward operator using formula~\ref{eq:forward-reception-dis} and the emission step of the adjoint operator using formula~\eqref{eq:adjoint-emission-dis}, the regularized Dirac delta distributions were nulled at sampled points far from the receiver nodes. Specifically, sampled points for which the value \(\bar{\delta}_b\) was less than \(\frac{0.005}{b^d}\) (with \(d = 2\)) were set to zero. Figure~\ref{fig:3b} shows a binary map indicating the sampled points that contribute to the approximation of the Dirac delta distributions across all receiver nodes.

The minimization is performed iteratively using the updates in~\eqref{eq:iterate}. The termination criterion~\eqref{eq:termination} is satisfied after 25 iterations, resulting in the final reconstructed image shown in Figure~\ref{fig:3d}. The values of the objective function are plotted against the iteration index in Figure~\ref{fig:3e}, demonstrating the monotone convergence of the algorithm. Figure~\ref{fig:3f} shows the relative error of the reconstructed images with respect to the ground truth and against the iteration index, defined as
\begin{align}
RE^n(\%) = \frac{\left\| \left( p_0^n\right)_{\mathrm{interpolated}} - \left( p_0\right)_{\mathrm{phantom}} \right\|_2}{\left\| \left( p_0\right)_{\mathrm{phantom}} \right\|_2} \times 100,
\end{align}
where \( \left( p_0^n \right)_{\mathrm{interpolated}} \) denotes the reconstructed initial pressure distribution at iteration \(n\), interpolated onto the grid used for the simulation of PAT measurements, \( \left( p_0 \right)_{\mathrm{phantom}} \) denotes the initial pressure distribution of the phantom used as the ground truth, and \(RE^n\) denotes the relative error at iteration \(n\).

\begin{figure} 
\centering
\subfigure[]{\includegraphics[width=0.45\textwidth]{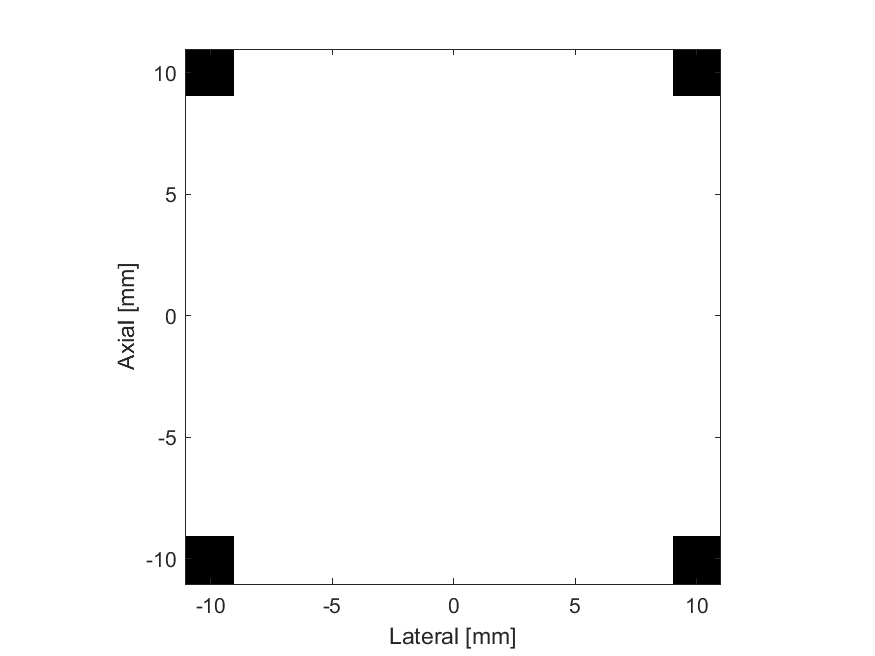}
\label{fig:3a}  }
\subfigure[]{\includegraphics[width=0.45\textwidth]{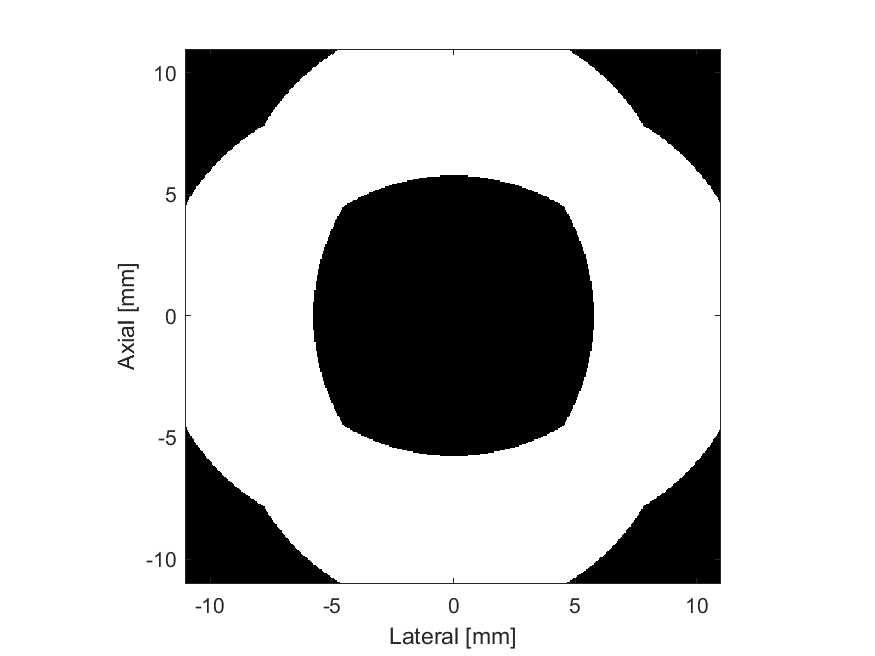}
\label{fig:3b} }
\subfigure[]{\includegraphics[width=0.45\textwidth]{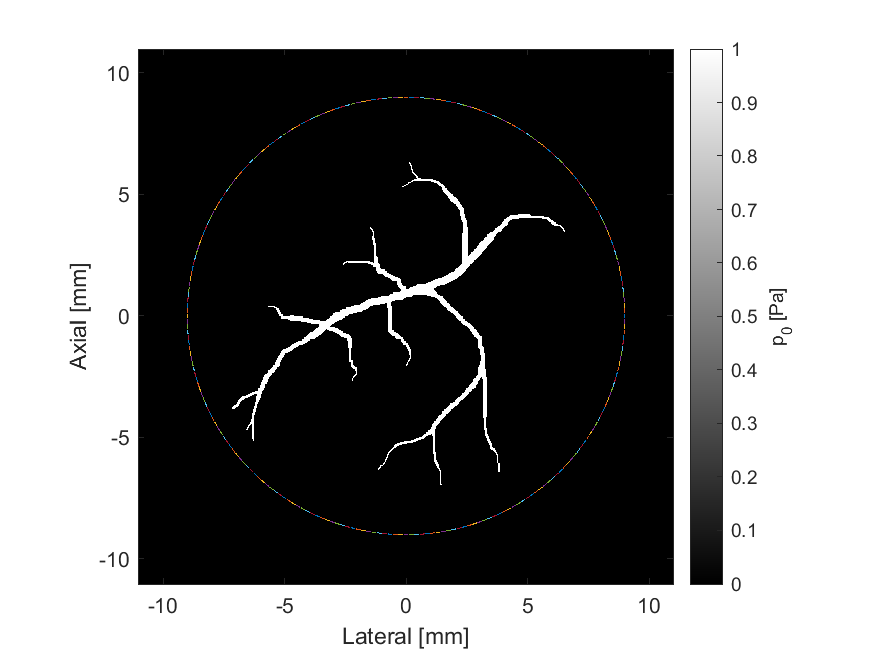}
\label{fig:3c} }
\subfigure[]{\includegraphics[width=0.45\textwidth]{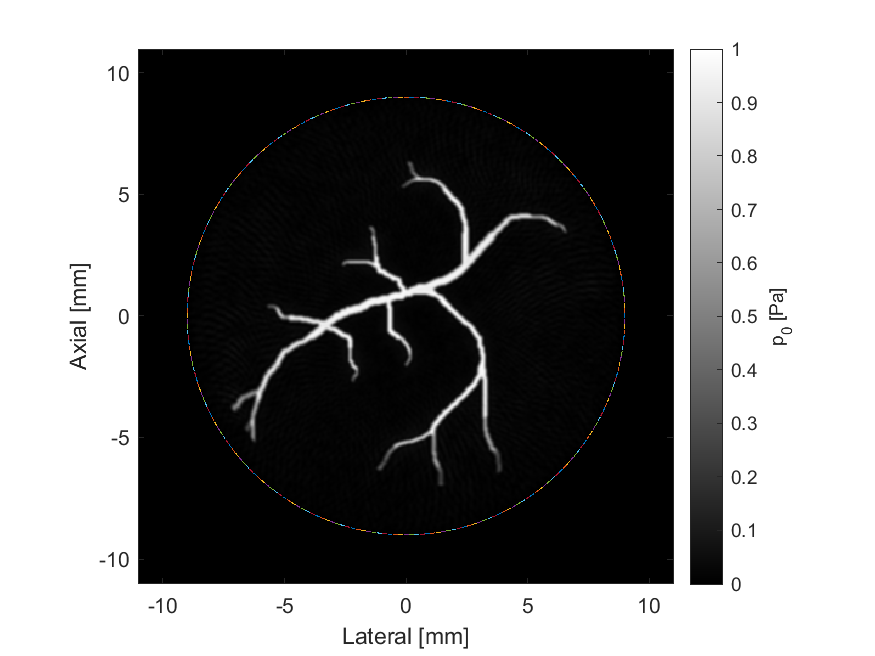}
\label{fig:3d} }
\subfigure[]{\includegraphics[width=0.45\textwidth]{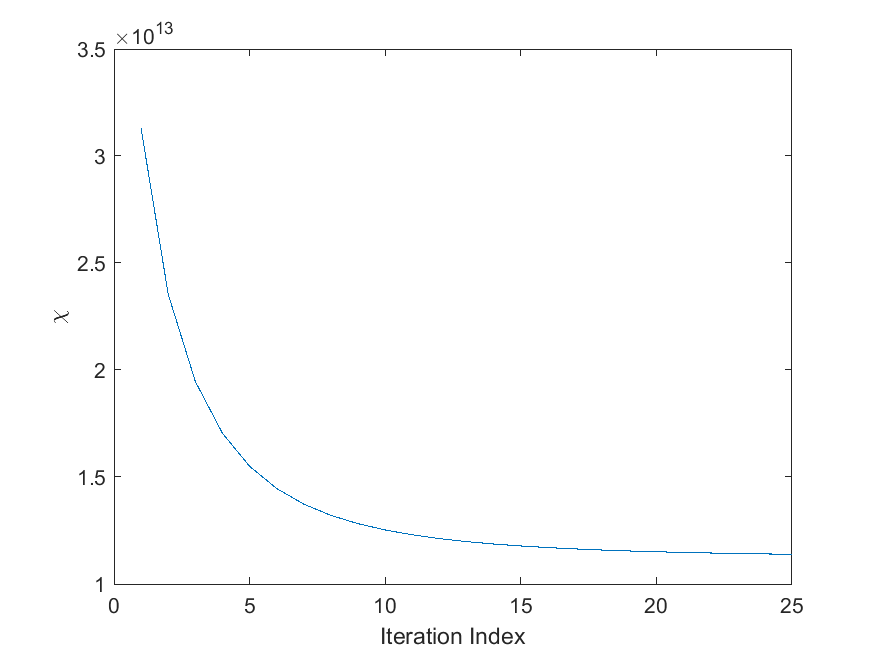}
\label{fig:3e} }
\subfigure[]{\includegraphics[width=0.45\textwidth]{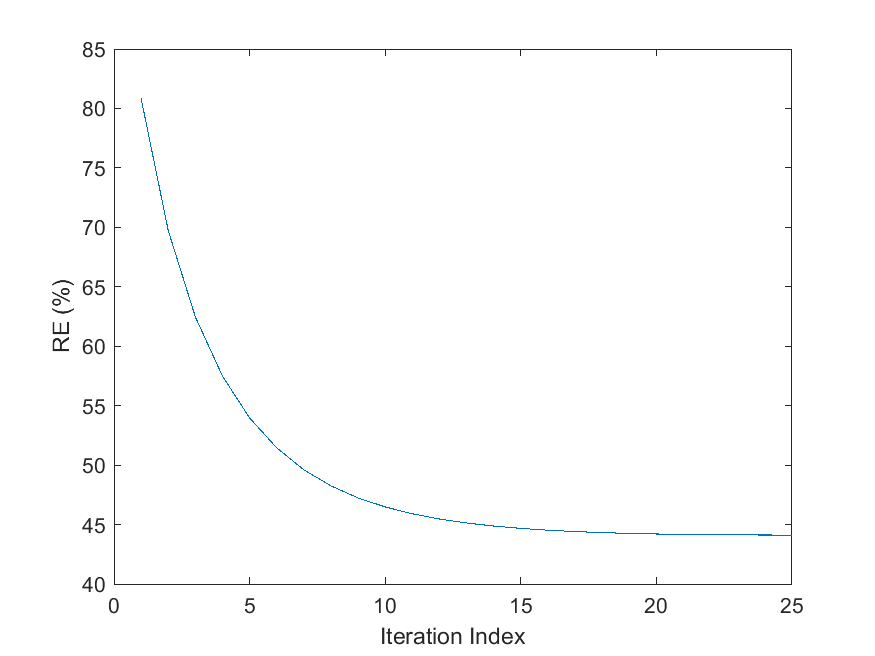}
\label{fig:3f} }
\caption{Binary mask of the sampled points contributing to the approximation of the Dirac delta distributions for all receiver nodes; black denotes false and white denotes true.: (a) simulation of measurements, (b) iterative image reconstruction. The initial pressure source distribution \(p_0\): (c) phantom used for the simulation of measurements (ground truth), (d) the final reconstructed image after 25 iterations. Algorithm convergence: (e) objective function values plotted against the iteration index; (f) relative error values plotted against the iteration index.}
\label{fig:3}
\end{figure}

\section{Discussion and Conclusion}

In this manuscript, we extended our recently proposed acoustic forward--adjoint pair of operators \cite{Javaherian2022} to photoacoustic tomography (PAT). The main contribution of the proposed acoustic forward--adjoint operator was a reception operator that maps, at each time step, the pressure wavefield in free space to Dirichlet-type data on the receiver surface \cite{Javaherian2022}. It was shown that this reception operator acts as a left-inverse of an emission operator that maps, at each time step, a Dirichlet-type source restricted to an emitter surface to the pressure wavefield in free space. The derived emission and reception operators satisfy the reciprocity law, which is a fundamental principle in physics. It was shown that, for a forward operator employing this reception operator, the corresponding adjoint operator is a time-reversed double-layer potential integral formula, which is equivalent to a time-reversal operator subject to the Dirichlet boundary condition.

On a Cartesian grid, this formulation expresses the pressure field restricted to a boundary surface as a smeared discontinuity in the field \cite{Javaherian2022}. It was shown that while the smeared singularity arising in the emission operator is canceled by numerical integration, the reception operator depends on the bandwidth of the regularized Dirac delta distribution used to regularize the singularity that arises when restricting the field to the surface \cite{Javaherian2022}. To avoid this non-uniqueness, we employ a scaled variant of the reception operator that removes this bandwidth dependency.

In this manuscript, the forward--adjoint operator framework was tailored to PAT. Accordingly, the PAT forward operator was expressed as a composite operator consisting of an acoustic operator equipped with the scaled reception operator proposed in \cite{Javaherian2022}, and an operator describing the photoacoustic source in terms of an initial pressure distribution. The PAT source operator is singular at the time origin and is therefore regularized in time.

An inner-product relation was then derived to obtain the adjoint of the PAT source operator, which was subsequently coupled with the acoustic adjoint operator to form the complete PAT adjoint operator. A discretized form of the inner-product relation for the full PAT forward--adjoint operator was verified using both on-grid and off-grid finite-sized receivers, demonstrating the mathematical consistency of the derived operator pair.

Finally, a photoacoustic scenario was constructed to demonstrate the performance of the proposed PAT forward--adjoint operator for reconstructing an image of the initial pressure distribution from Dirichlet-type boundary data within an iterative minimization framework. The results show that the reconstructed image is consistent with the ground truth both qualitatively and quantitatively. Future work will focus on translating the proposed PAT forward--adjoint operator framework toward clinical applications.

\bibliography{references}
\bibliographystyle{spiebib}

\end{document}